\documentclass[12pt]{article}
\usepackage{enumerate}
\usepackage{amsthm,amssymb,amsmath,mathrsfs}
\usepackage{enumitem,mathtools}
\usepackage{url}
\usepackage{fullpage}
\usepackage{hyperref}
\usepackage{graphicx}
\usepackage{array}
\usepackage{multirow}
\usepackage{longtable}
\usepackage{tikz}
\usepackage{bbm}
\usepackage{bm}
\makeatletter
\newcommand*{\transpose}{
  {\mathpalette\@transpose{}}
}
\newcommand*{\@transpose}[2]{
  \raisebox{\depth}{$\m@th#1\intercal$}
}

\newtheorem{thm}{Theorem}[section]
\newtheorem{dfn}[thm]{Definition}

\newtheorem{cor}[thm]{Corollary}
\newtheorem{lem}[thm]{Lemma}
\newtheorem{prop}[thm]{Proposition}

\DeclareMathOperator{\rank}{rank}
\DeclareMathOperator{\Spec}{Spec}
\DeclareMathOperator{\tr}{tr}

\title{Real equiangular lines in dimension 18 and the \\
Jacobi identity for complementary subgraphs}

\author{Gary R.W. Greaves 
  \thanks{School of Physical and Mathematical Sciences, 
  Nanyang Technological University, 
  21 Nanyang Link, Singapore 637371,
 \tt{gary@ntu.edu.sg}}
 \and
 Jeven Syatriadi 
  \thanks{School of Physical and Mathematical Sciences, 
  Nanyang Technological University, 
  21 Nanyang Link, Singapore 637371,
 \tt{jeve0002@e.ntu.edu.sg}}
}

\date{}

\begin{document}

\maketitle

\begin{abstract}
 We show that the maximum cardinality of an equiangular line system in $\mathbb R^{18}$ is at most $59$.
 Our proof includes a novel application of the Jacobi identity for complementary subgraphs.
 In particular, we show that there does not exist a graph whose adjacency matrix has characteristic polynomial $(x-22)(x-2)^{42} (x+6)^{15} (x+8)^2$.
\end{abstract}

\section{Introduction}

A \textbf{real equiangular line system} is a set of lines through the origin in Euclidean space where the angle between any pair of lines is the same.
A classical problem in elliptical geometry is to determine the maximum number $N(d)$ of equiangular lines in Euclidean space $\mathbb{R}^d$ for a fixed dimension $d$.
In general, computing $N(d)$ is a difficult problem dating back to a paper of Haantjes \cite{haantjes48} in 1948.
In the 1970s, Seidel et al.~\cite{lemmens73, vLintSeidel66, Seidel74} made important progress on the study of real equiangular lines, including the discovery of Gerzon's absolute bound: $N(d) \leqslant d(d+1)/2$.
Furthermore, in 1973 \cite{lemmens73}, the value of $N(d)$ was settled for all $d \leqslant 13$, $d=15$, and $d = 21,22,23$.
In 2000, by giving an infinite construction of large equiangular line systems, De~Caen~\cite{decaen2000} showed that $N(d) = \Theta(d^2)$.

In order to study $N(d)$, it is important to understand $N_\alpha(d)$: the maximum number of equiangular lines in $\mathbb R^d$ with a fixed common angle $\arccos (\alpha)$ where $\alpha \in (0,1)$.
Using Neumann's observation \cite{lemmens73}, we know that if $N_{\alpha}(d)$ is greater than $2d$ then $\alpha$ must be equal to $1/a$ where $a$ is an odd integer.
Lemmens and Seidel~\cite{lemmens73} showed that $N_{1/3}(d) = 28$ for $7 \leqslant d \leqslant 15$ and $N_{1/3}(d) = 2(d-1)$ for $d \geqslant 15$.
Within the past decade, there has been great progress in the study of equiangular lines and its variants, involving a wide array of techniques from many areas of pure and applied mathematics.
Cao, Koolen, Lin, and Yu~\cite{CKLY22} showed that $N_{1/5}(d) = 276$ for $23 \leqslant d \leqslant 185$ and $N_{1/5}(d) = \left\lfloor \frac{3(d-1)}{2} \right\rfloor$ for $d \geqslant 185$.
Asymptotically, this pattern continues.
Indeed, Jiang, Tidor, Yao, Zhang, and Zhao~\cite{annmath}, showed that when $a \geqslant 3$ is an odd integer, $\alpha = 1/a$, and $d$ is large enough, $N_\alpha(d) = \left\lfloor \frac{a+1}{a-1}(d-1) \right\rfloor$, thereby proving a recent conjecture of Jiang and Polyanskii~\cite{jiangpolyanskii}.
Other related progress includes improvements to the relative bound and various other useful constraints for $N_\alpha(d)$~\cite{balla21, GlazYu, KaoYu, KingTang, delaat, okudayu16, Yu17}.
In low dimensional Euclidean spaces, substantial improvements to upper and lower bounds for $N(d)$ have been achieved through new spectral methods and new constructions of equiangular lines \cite{GG18, GKMS16, GSY21, GSYdim1718, GreavesYatsyna19, LinYu, Szo}.
We refer to the exposition of Kao and Yu~\cite{KaoYu} for a more in depth overview of the recent history of the problem.

Recently, the authors together with Yatsyna \cite{GSY21, GSYdim1718} solved the problem of determining $N(14)$, $N(16)$, and $N(17)$.
The smallest $d$ for which $N(d)$ is currently unknown is $d=18$.
In a series of papers from the last few years \cite{GSYdim1718, LinYu, Szo}, the lower bound for $N(18)$ has increased from $48$ to $57$.
Augmenting this recent progress, our main result of this paper is a new upper bound for $N(18)$:
\begin{thm}
\label{thm:main}
$N(18) \leqslant 59$.
\end{thm}

We obtain Theorem~\ref{thm:main} by studying the \emph{Seidel matrices} that correspond to systems of equiangular lines as follows.
Suppose $\{ \mathbf{v}_1,\dots, \mathbf{v}_{n} \}$ is a set of unit spanning vectors for an equiangular line system of cardinality $n$ in $\mathbb R^{d}$ with $n > d$.
For all $i \ne j$, the inner product of $\mathbf v_i$ and $\mathbf v_j$ is equal to $\pm \alpha$ for some $\alpha \in (0,1)$.
The \textbf{Seidel matrix} $S$ corresponding to this line system is defined as $S = (G-I)/\alpha$ where $G$ is the Gram matrix for the set of vectors $\{ \mathbf{v}_1,\dots, \mathbf{v}_{n} \}$.

We begin by enumerating all possible polynomials that could be the characteristic polynomial $\operatorname{Char}_S(x) \coloneqq \det(xI-S)$ of a Seidel matrix that corresponds to a system of $60$ equiangular lines in $\mathbb R^{18}$.
We then apply the techniques developed in \cite{GSY21, GSYdim1718} to rule out the existence of Seidel matrices whose characteristic polynomial is equal to any of the candidate characteristic polynomials enumerated above.
However, these techniques are not quite strong enough to prove Theorem~\ref{thm:main} - one last candidate characteristic polynomial $(x+5)^{42} (x-11)^{15} (x-15)^3$ remained resilient to all known methods of establishing the nonexistence of a corresponding Seidel matrix.
Thus, the novelty in this paper is to apply the \emph{Jacobi identity} from \cite{decaen} to develop a method to rule out this one last possibility.

As shown in \cite{GG18}, the existence of a Seidel matrix with characteristic polynomial equal to $(x+5)^{42} (x-11)^{15} (x-15)^3$ is equivalent to the existence of a (regular) graph with characteristic polynomial $(x-22)(x-2)^{42} (x+6)^{15} (x+8)^2$.
In a spectral sense, such a graph is rather close to being strongly regular.
It tends to be notoriously difficult to establish nonexistence results for strongly regular graphs.
One can observe that the earlier improvements to the upper bounds for $N(17)$, $N(19)$, and $N(20)$ relied on elaborate nonexistence results for certain strongly regular graphs~\cite{Azarija75, Azarija95, srg49}.

In Table~\ref{tab:newbound18}, we give the latest update on the values of lower and upper bounds for $N(d)$ where $d \leqslant 43$, including the improvement from this paper.
(See Sequence A002853 in The On-Line Encyclopedia of Integer Sequences~\cite{oeis}.)
For a more extensive history on the developments of the bounds, we again refer the reader to \cite{KaoYu}.
It is worth noting that for $d \leqslant 41$, the only remaining unknown values of $N(d)$ are $N(18)$, $N(19)$, and $N(20)$.
\begin{table}[ht]
	\begin{center}
	\setlength{\tabcolsep}{2pt}
	\begin{tabular}{c|ccccccccccccccccc}
		$d$  & 2 & 3        & 4           & 5  & 6  & 7--14 & 15 & 16 & 17 & 18 & 19 & 20 & 21 & 22 & 23 -- 41 & 42 & 43 \\\hline
		$N(d)$  & 3 & 6        & 6           & 10 & 16 & 28  & 36 & 40 & 48 & 57--59 & 72--74 & 90--94  & 126 & 176 & 276 & 276 -- 288 & 344
	\end{tabular}
	\end{center}
	\caption{Bounds for the sequence $N(d)$ for $2\leqslant d\leqslant 43$, including the improvement from this paper.}
	\label{tab:newbound18}
\end{table}

The outline of the paper is as follows.
In Section~\ref{sec:cand_char_poly}, using the so-called polynomial enumeration algorithm of \cite{GSY21, GSYdim1718}, we enumerate all $44$ possible candidate characteristic polynomials for a Seidel matrix corresponding to a system of $60$ equiangular lines in $\mathbb R^{18}$.
We employ methods from \cite{GSY21, GSYdim1718} to show that 43 out of these 44 candidate characteristic polynomials cannot be the characteristic polynomial of any Seidel matrix.
The computational parts of these methods ran in Magma~\cite{magma} and Mathematica~\cite{mathematica}.
The total running time of all the computations used in this paper is less than 40 minutes running on a modern PC.
The Magma~\cite{magma} implementation of all the computations in this paper is available on GitHub~\cite{github}.
Section~\ref{sec:decaen_jacobi} is devoted to ruling out the sole surviving candidate characteristic polynomial $(x+5)^{42} (x-11)^{15} (x-15)^3$.
There, we use the Jacobi identity for complementary subgraphs and derive necessary algebraic conditions for a graph to be an induced subgraph of $\Gamma$, a regular graph in the switching class of a putative Seidel matrix whose characteristic polynomial is $(x+5)^{42} (x-11)^{15} (x-15)^3$.

\section{Candidate characteristic polynomials}
\label{sec:cand_char_poly}

\subsection{Enumerating candidates}

In this section, we find all the candidate characteristic polynomials of a Seidel matrix that corresponds to an equiangular line system of cardinality $60$ in $\mathbb R^{18}$.
See Theorem~\ref{thm:candpols}.

By \cite[Lemma 3.1]{GSY21}, we derive the following lemma where $\kappa = 11$ and $\theta = 13$:
\begin{lem}
\label{lem:ev_mult}
Let $S$ be a Seidel matrix corresponding to $60$ equiangular lines in $\mathbb R^{18}$. Then 
\[
    \operatorname{Char}_S(x) = (x+5)^{42}(x-11)^6 \phi(x),
\]
for some monic polynomial $\phi$ of degree 12 in $\mathbb{Z}[x]$ all of whose zeros are greater than $-5$.
\end{lem}

Let $S$ be a Seidel matrix corresponding to an equiangular line system of cardinality $60$ in $\mathbb R^{18}$.
The next step is to find feasible polynomials for $\phi(x)=\sum_{t=0}^{12} b_t x^{12-t}$ where
$$
    \operatorname{Char}_S(x) = (x+5)^{42}(x-11)^6 \phi(x).
$$
Clearly $b_0 = 1$.
We can find $b_1$ and $b_2$, using $\tr S=0$ and $\tr S^2=60\cdot 59 = 3540$ together with Newton's identities: $b_1=-144$, and $b_2=9486$.
Thus, we need to find all totally-real, integer polynomials $\phi(x)$ such that $b_0=1$, $b_1=-144$, and $b_2=9486$.
We note that the top three coefficients of $\phi(x-1)$ are: $1$, $-156$, and $11136$ respectively.

Let $p(x)=\sum_{i=0}^n a_i x^{n-i}$ be a monic polynomial in $\mathbb{Z}[x]$.
Following~\cite{GSY21}, we say $p(x)$ is \textbf{type $\mathbf 2$} if $2^i$ divides $a_i$ for all $i\geqslant 0$ and \textbf{weakly type $\mathbf 2$} if $2^{i-1}$ divides $a_i$ for all $i\geqslant 1$.
By {\cite[Lemma 2.7 and Lemma 2.8]{GSY21}}, the polynomial $\phi(x-1)$ must be type $2$.
Furthermore, just as in \cite{GSY21, GSYdim1718}, we employ the polynomial enumeration algorithm of \cite[Section 2.3]{GSY21} to establish the following theorem.
We refer to \cite{GSY21} for details.

\begin{thm}
\label{thm:candpols}
Let $S$ be a Seidel matrix corresponding to $60$ equiangular lines in $\mathbb R^{18}$.
Then 
\begin{itemize}
    \item[(i)] $\operatorname{Char}_S(x)$ is one of the 39 polynomials listed in Table~\ref{tab:39polydim18},

    \item[(ii)] $\displaystyle 
         \operatorname{Char}_S(x) \in  \left \{ \begin{array}{l}
            (x+5)^{42} (x-9)^3 (x-11)^6 (x-13)^9, \\
            (x+5)^{42} (x-11)^{14} (x-13)^3 (x-17), \\
            (x+5)^{42} (x-9)^2 (x-11)^9 (x-13)^6 (x-15), \\
            (x+5)^{42} (x-11)^{10} (x-13)^6 (x^2-22x+109)
        \end{array}
        \right \},$
    \item[(iii)] or $\operatorname{Char}_S(x) = (x+5)^{42} (x-11)^{15} (x-15)^3 $.
    \end{itemize}
\end{thm}

The rest of the paper is dedicated to ruling out the existence of Seidel matrices whose characteristic polynomial is equal to any of the 44 given in Theorem~\ref{thm:candpols}.

\subsection{Certificates of infeasibility}

First, we rule out 39 of the 44 candidate characteristic polynomials in Theorem~\ref{thm:candpols}, using certificates of infeasibility, which we define below. 

Denote by $\mathscr S_n$ the set of all Seidel matrices of order $n$.
Given a positive integer $e$, define the set $\mathscr P_{n,e} = \{ \operatorname{Char}_S(x) \mod 2^e \mathbb Z[x] \; : \; S \in \mathscr S_n \}$.
We will require the following upper bound on the cardinality of $\mathscr P_{n,e}$ for odd $n$.

\begin{thm}[{\cite[Corollary 3.13]{GreavesYatsyna19}}]\label{thm:countCharPolySeidelOdd}
			Let $n$ be an odd integer and $e$ be a positive integer.
			Then the cardinality of $\mathscr P_{n,e}$ is at most $2^{\binom{e-2}{2}+1}$.
\end{thm}

For $n=59$ and $e=7$, randomly generating Seidel matrices of order $59$ is sufficient to obtain $2048$ characteristic polynomials in distinct congruence classes modulo $2^e \mathbb Z[x]$.
Since $2048 = 2^{\binom{e-2}{2}+1}$ is the maximum possible number of elements of $\mathscr{P}_{59,7}$, it follows that $|\mathscr{P}_{59,7}|=2048$.
Hence, we have explicitly constructed the set $\mathscr{P}_{59,7}$ and we will use it in Definition~\ref{dfn:interlacing} below.

For $M$ a real symmetric matrix of order $n$ and $\mathcal{N}$ a subset of $\{1,2,\dots,n\}$, we denote by $M[\mathcal{N}]$ the principal submatrix of $M$ formed by the rows and columns that are indexed by the elements of $\mathcal{N}$.
We use $\overline{\mathcal N}$ to denote the set complement of $\mathcal N$, i.e., $\overline{\mathcal{N}} \coloneqq \{1,\dots,n\} \backslash \mathcal{N}$.
Next we introduce the notion of interlacing, motivated by Cauchy's interlacing theorem:

		\begin{thm}[\cite{Cau:Interlace,Fisk:Interlace05,Hwang:Interlace04}]\label{thm:cauchyinterlace}
			Let $M$ be a real symmetric matrix of order $n$ having eigenvalues $\lambda_1 \leqslant \lambda_2 \leqslant \dots \leqslant \lambda_n$ and suppose $M[\overline{\{i\}}]$, for some $i \in \{1,\dots,n \}$, has eigenvalues $\mu_1 \leqslant \mu_2 \leqslant \dots \leqslant \mu_{n-1}$.
			Then
			\[
				\lambda_1 \leqslant \mu_1 \leqslant \lambda_2 \leqslant \dots \leqslant \lambda_{n-1} \leqslant \mu_{n-1} \leqslant \lambda_n.
			\]
		\end{thm}

Let $f(x) = \prod_{i=0}^{e}(x-\lambda_i)$ and $g(x) = \prod_{i=1}^{e}(x-\mu_i)$ such that $\lambda_0 \leqslant \lambda_1 \leqslant \dots \leqslant \lambda_e$, and $\mu_1 \leqslant \mu_2 \leqslant \dots \leqslant \mu_e$.
We say that $g$ \textbf{interlaces} $f$ if $\lambda_0 \leqslant \mu_1 \leqslant \lambda_1 \leqslant \dots \leqslant \mu_e \leqslant \lambda_e$.

\begin{dfn}
\label{dfn:interlacing}
Let $p(x) \in \mathbb Z[x]$ be a monic polynomial of degree $n$ and write
\[
p(x)=\sum_{t=0}^{n} a_t x^{n-t},
\]
where $a_0 = 1$, $a_1 = 0$, and $a_2= -\binom{n}{2}$. 
An \textbf{interlacing characteristic polynomial} for $p(x)$ is defined to be a totally-real, integer polynomial $\mathfrak{f}(x)=\sum_{t=0}^{n-1} b_t x^{n-1-t}$ such that
\begin{enumerate}[label=(\roman*)]
\item $b_0=1$, $b_1=0$, $b_2= -\binom{n-1}{2}$,
\item $\mathfrak{f}(x)$ interlaces $p(x)$,
\item $\mathfrak{f}(x-1)$ is weakly type 2 and is type 2 if $n-1$ is even,
\item $\mathfrak{f}(x)$ is in a congruence class of $\mathscr{P}_{n-1,7}$, if $n-1$ is odd.
\end{enumerate}
\end{dfn}

Denote by $\operatorname{Deck}(p)$ the set of all interlacing characteristic polynomials for $p(x)$.
The next lemma follows from \cite[Theorem 5.1]{GSYdim1718}.

\begin{lem}
\label{lem:interlacingconfig}
Suppose $p(x)$ is the characteristic polynomial of a Seidel matrix $S$.
Then there exist nonnegative integers $n_{\mathfrak f}$ for each $\mathfrak f(x) \in \operatorname{Deck}(p)$ such that 
\begin{equation}\label{eq:sumofsubpoly}
\sum_{\mathfrak f(x) \in \operatorname{Deck}(p)} n_{\mathfrak f} \cdot \mathfrak f(x)= p^\prime(x).
\end{equation}
\end{lem}

The (row) vector $\mathbf n$ indexed by $\operatorname{Deck}(p)$ whose $\mathfrak f(x)$-entry is $n_{\mathfrak f}$, for each $\mathfrak f(x) \in \operatorname{Deck}(p)$ is called an \textbf{interlacing configuration for $p(x)$.}
We can use Lemma~\ref{lem:interlacingconfig} to show that there does not exist a Seidel matrix having $p(x)$ as its characteristic polynomial using information about interlacing configurations.

The \textbf{coefficient vector} of a polynomial $f(x)=\sum_{t=0}^{n-1} c_t x^{n-1-t}$ of degree $n-1$ is defined to be the (row) vector $(c_0,c_1,\dots,c_{n-1})$.
Given a set $\mathfrak P$ of polynomials each of degree $n - 1$, the \textbf{coefficient matrix} $\operatorname{Coeff}(\mathfrak P)$ is defined as the $|\mathfrak P| \times n$ matrix whose rows are the coefficient vectors for each polynomial in $\mathfrak P$.
If $\mathfrak P$ is a singleton, i.e.,  $\mathfrak P =\{ \mathfrak f (x) \}$ then we merely write $\operatorname{Coeff}(\mathfrak P)$ as $\operatorname{Coeff}(\mathfrak f)$ and if its columns require a particular order (e.g., in Lemma~\ref{lem:last5_quad109}) then we write $\operatorname{Coeff}(\mathfrak f_1, \dots, \mathfrak f_k)$ for $\operatorname{Coeff}(\mathfrak P)$ where $\mathfrak P = \{\mathfrak f_1(x), \dots, \mathfrak f_k(x)\}$.
Note that \eqref{eq:sumofsubpoly} can be written as a vector equation as
\begin{equation}
\label{eqn:sumofsubpolyVec}
    \mathbf n \cdot \operatorname{Coeff}(\operatorname{Deck}(p)) = \operatorname{Coeff}(p^\prime).
\end{equation}

We write $\Lambda_p$ for the set of distinct zeros of the polynomial $p(x)$ and define the polynomial $\operatorname{Min}_p(x) \coloneqq \prod_{\lambda \in \Lambda_p} (x-\lambda)$.
By Definition~\ref{dfn:interlacing}, if $\mathfrak{f}(x) \in \operatorname{Deck}(p)$ then $\mathfrak{f}(x)$ interlaces $p(x)$ and the top three coefficients of $\mathfrak{f}(x)$ are fixed (see (i) of Definition~\ref{dfn:interlacing}).
Consequently, we have the following lemma, which is similar to Lemma 5.4 in \cite{GreavesYatsyna19}.

\begin{lem}\label{lem:coeff_of_interlacepoly}
Let $p(x)$ be a polynomial in $\mathbb Z[x]$.
Suppose $\operatorname{Min}_p(x) = \sum_{i=0}^e a_i x^{e-i}$.
Then, for all $\mathfrak f(x) \in \operatorname{Deck}(p)$,
\[
    \mathfrak f(x)=\dfrac{p(x)}{\operatorname{Min}_p(x)}\sum_{i=0}^{e-1} b_i x^{e-1-i},
\]
where $b_0 = 1$, $b_1 = a_1$, $b_2 = a_2+n-1$, and $b_i \in \mathbb Z$ for $i \in \{3,\dots,e-1\}$.
\end{lem}
Lemma~\ref{lem:coeff_of_interlacepoly} above enables us to use the polynomial enumeration algorithm of \cite[Section 2.3]{GSY21} to construct $\operatorname{Deck}(p)$, similar to what we did in \cite{GSY21, GSYdim1718}.

We write $\mathbf x \geqslant \mathbf 0$ to indicate that all entries of the vector $\mathbf x$ are nonnegative.
In the following corollary, note that the polynomial $p(x)$ divides $\operatorname{Min}_p(x)p^\prime (x)$ and $\operatorname{Min}_p(x)\mathfrak f(x)$ for each $\mathfrak f(x) \in \operatorname{Deck}(p)$.

\begin{cor}
\label{cor:infeasibility}
Let $p(x)$ be a polynomial where $\operatorname{Min}_p(x)$ has degree $e$.
Suppose there exists a vector $\mathbf c \in \mathbb R^{e}$ such that
\[
\operatorname{Coeff}\left (\left \{ \dfrac{\operatorname{Min}_p(x) \mathfrak f(x) }{p(x)} \; : \; \mathfrak f(x) \in \operatorname{Deck}(p) \right \}\right )\mathbf{c} \geqslant\mathbf{0} \quad \text{and} \quad \operatorname{Coeff} \left ( \left \{\dfrac{\operatorname{Min}_p(x)p^\prime(x)}{p(x)}  \right \} \right )\mathbf{c} <0.
\]
Then $p(x)$ is not the characteristic polynomial of any Seidel matrix.
\end{cor}
\begin{proof}
By Lemma~\ref{lem:interlacingconfig}, Lemma~\ref{lem:coeff_of_interlacepoly}, and Farkas' Lemma (see \cite[Theorem 4.1]{GSY21}), it follows that $p(x)$ does not have an interlacing configuration.
\end{proof}

We call the vector $\mathbf c$ from Corollary~\ref{cor:infeasibility} a \textbf{certificate of infeasibility} for $p(x)$.
In this paper, certificates of infeasibility are written as tuples.
Now we can rule out $39$ of the $44$ candidate characteristic polynomials of Theorem~\ref{thm:candpols}.

\begin{lem}
\label{lem:firstTable}
There does not exist a Seidel matrix $S$ whose characteristic polynomial is equal to any of the polynomials in Table~\ref{tab:39polydim18}.
\end{lem}
\begin{proof}
    Each polynomial in Table~\ref{tab:39polydim18} is listed together with a certificate of infeasibility.
\end{proof}

Now we continue by ruling out the candidate characteristic polynomials from part (ii) of Theorem~\ref{thm:candpols} using techniques from \cite{GSYdim1718}.

\subsection{Multiplicity of interlacing characteristic polynomials}

The next lemma puts restrictions on the entries of an interlacing configuration.

\begin{lem}[{\cite[Lemma 6.2]{GSYdim1718}}] \label{lem:extracting}
Let $M$ be a real symmetric matrix of order $n$ with eigenvalue $\lambda$ of multiplicity $m$.
Suppose there exists a $k$-subset $\mathcal{I}$ of $\{1,\dots,n\}$ such that, for each $i \in \mathcal{I}$, the principal submatrix $M[\overline{\{i\}}]$ has an eigenvalue $\lambda$ of multiplicity $m$ or $m+1$.
Then $M$ has a principal submatrix of order $n-k$ with eigenvalue $\lambda$ of multiplicity at least $m$.
\end{lem}

Now we employ Lemma~\ref{lem:extracting} to rule out the existence of a Seidel matrix corresponding to either of two of the candidate characteristic polynomials from part (ii) of Theorem~\ref{thm:candpols}.

\begin{lem} \label{lem:last5_fourint}
There does not exist a Seidel matrix $S$ with characteristic polynomial
$$
\operatorname{Char}_S(x) = (x+5)^{42} (x-9)^3 (x-11)^6 (x-13)^9.
$$
\end{lem}

\begin{proof}
Suppose a Seidel matrix $S$ has characteristic polynomial
$$
\operatorname{Char}_S(x) = (x+5)^{42} (x-9)^3 (x-11)^6 (x-13)^9.
$$
Then $\operatorname{Deck}(\operatorname{Char}_S) = \{\mathfrak f_1(x),\mathfrak f_{2}(x)\}$, where
\begin{align*}
    \mathfrak{f}_1(x) &= (x+5)^{41} (x-6) (x-9)^2 (x-11)^7 (x-13)^8, \\
    \mathfrak{f}_2(x) &= (x+5)^{41} (x-9)^3 (x-11)^5 (x-13)^8 (x^2-19x+82).
\end{align*}
and we find that there is only one possible interlacing configuration $(n_{\mathfrak f_1},n_{\mathfrak f_2}) = (28,32)$.
By Lemma~\ref{lem:extracting}, the Seidel matrix $S$ has a principal submatrix $T$ of order 32 with eigenvalue $11$ of multiplicity at least $6$.
Furthermore, by Theorem~\ref{thm:cauchyinterlace}, we also have that $\operatorname{Char}_{T}(x)$ is divisible by $(x+5)^{14}$.
Therefore, $\operatorname{Char}_{T}(x)$ is divisible by $(x+5)^{14} (x-11)^6$.
This is a contradiction since $\tr T^2 = 32 \cdot 31 = 992 < 1076 = 14 \cdot (-5)^2 + 6 \cdot 11^2$.
\end{proof}

\begin{lem} \label{lem:last5_ev17}
There does not exist a Seidel matrix $S$ with characteristic polynomial
$$
\operatorname{Char}_S(x) = (x+5)^{42} (x-11)^{14} (x-13)^3 (x-17).
$$
\end{lem}

\begin{proof}
Suppose a Seidel matrix $S$ has characteristic polynomial
$$
\operatorname{Char}_S(x) = (x+5)^{42} (x-11)^{14} (x-13)^3 (x-17).
$$
Then $\operatorname{Deck}(\operatorname{Char}_S) = \{\mathfrak f_1(x),\mathfrak f_{2}(x), \mathfrak f_{3}(x)\}$, where
\begin{align*}
    \mathfrak{f}_1(x) &= (x+5)^{41} (x-11)^{13} (x-13)^3 (x^2-23x+106), \\
    \mathfrak{f}_2(x) &= (x+5)^{41} (x-11)^{13} (x-13)^2 (x^3-36x^2+405x-1382), \\
    \mathfrak{f}_3(x) &= (x+5)^{41} (x-11)^{13} (x-13)^2 (x-17) (x^2-19x+82).
\end{align*}
The polynomial $\mathfrak{f}_2(x)$ has $105$ interlacing characteristic polynomials but it has a corresponding certificate of infeasibility $(51115513410, 0, 0, 6301221, 484710, 37285)$.
Therefore, there does not exist a Seidel matrix with characteristic polynomial $\mathfrak{f}_2(x)$.
Furthermore, there is only one possible interlacing configuration $(n_{\mathfrak{f}_1},n_{\mathfrak{f}_2},n_{\mathfrak{f}_3})=(33,0,27)$ where $n_{\mathfrak{f}_2}=0$.
By Lemma~\ref{lem:extracting}, the Seidel matrix $S$ has a principal submatrix $T$ of order 27 with eigenvalue $13$ of multiplicity at least 3.
Furthermore, by Theorem~\ref{thm:cauchyinterlace}, we also have that $\operatorname{Char}_{T}(x)$ is divisible by $(x+5)^9$.
Therefore, $\operatorname{Char}_{T}(x)$ is divisible by $(x+5)^9(x-13)^3$.
This is a contradiction since $\tr T^2 = 27 \cdot 26 = 702 < 732 = 9 \cdot (-5)^2 + 3 \cdot 13^2$.
\end{proof}

\subsection{Warranted polynomials}

In this section, we will use the notion of a \emph{warranted} polynomial.
Let $p(x)$ be a polynomial.
Suppose that $p(x)$ has at least one interlacing configuration and $\operatorname{Min}_p(x)$ has degree $e$.
We say that $\mathfrak f(x) \in \operatorname{Deck}(p)$ is $p(x)$\textbf{-warranted} if the $\mathfrak f(x)$-entry of every interlacing configuration for $p(x)$ is positive.

Let $\mathfrak f(x) \in \operatorname{Deck}(p)$.
By Farkas' Lemma (see \cite[Theorem 4.1]{GSY21}), if there exists $\mathbf c \in \mathbb R^e$ such that the $\mathfrak h(x)$-entry of
\[
\operatorname{Coeff}\left ( \left \{ \dfrac{\operatorname{Min}_p(x) \mathfrak h(x) }{p(x)} \;  : \; \mathfrak h(x) \in \operatorname{Deck}(p) \right \}\right )\mathbf{c}
\]
is negative for $\mathfrak h(x) = \mathfrak f(x)$, nonnegative for $\mathfrak h(x) \in \operatorname{Deck}(p) \backslash \{\mathfrak f(x)\}$, and   
\[
 \operatorname{Coeff} \left ( \left \{\dfrac{\operatorname{Min}_p(x)p^\prime(x)}{p(x)}  \right \} \right )\mathbf{c} <0,
\]
then $\mathfrak f(x)$ is $p(x)$-warranted.
The vector $\mathbf c$ is called the \textbf{certificate of warranty} for $\mathfrak f(x)$.
In this paper, certificates of warranty are written as tuples.
The next lemma follows directly from the definition of warranted polynomials.

\begin{lem}\label{lem:warranted}
Let $S$ be a Seidel matrix of order $n$.
Suppose that $\mathfrak f(x) \in \operatorname{Deck}(\operatorname{Char}_S)$ is $\operatorname{Char}_S(x)$-warranted.
Then there exists $i \in \{1,\dots,n\}$ such that $\operatorname{Char}_{S[i]}(x) = \mathfrak f(x)$.
\end{lem}

Now we employ Lemma~\ref{lem:warranted} to rule out the existence of a Seidel matrix corresponding to one of the candidate characteristic polynomials from part (ii) of Theorem~\ref{thm:candpols}.

\begin{lem} \label{lem:last5_fiveint}
There does not exist a Seidel matrix $S$ with characteristic polynomial
$$
\operatorname{Char}_S(x) = (x+5)^{42} (x-9)^2 (x-11)^9 (x-13)^6 (x-15).
$$
\end{lem}

\begin{proof}
Suppose a Seidel matrix $S$ has characteristic polynomial
$$
\operatorname{Char}_S(x) = (x+5)^{42} (x-9)^2 (x-11)^9 (x-13)^6 (x-15).
$$
Then $\operatorname{Deck}(\operatorname{Char}_S) = \{\mathfrak f_1(x), \dots, \mathfrak f_7(x)\}$, where
\begin{align*}
    \mathfrak{f}_1(x) &= (x+5)^{41} (x-9) (x-11)^8 (x-13)^5 (x^4-43x^3+673x^2-4529x+11026), \\
    \mathfrak{f}_2(x) &= (x+5)^{41} (x-9) (x-11)^8 (x-13)^5 (x^4-43x^3+673x^2-4525x+10966), \\
    \mathfrak{f}_3(x) &= (x+5)^{41} (x-9)^2 (x-11)^8 (x-13)^6 (x^2-21x+94), \\
    \mathfrak{f}_4(x) &= (x+5)^{41} (x-9) (x-11)^8 (x-13)^5 (x^4-43x^3+673x^2-4513x+10818), \\
    \mathfrak{f}_5(x) &= (x+5)^{41} (x-9) (x-11)^9 (x-13)^5 (x^3-32x^2+321x-978), \\
    \mathfrak{f}_6(x) &= (x+5)^{41} (x-9) (x-10) (x-11)^8 (x-13)^6 (x^2-20x+83), \\
    \mathfrak{f}_7(x) &= (x+5)^{41} (x-9) (x-11)^9 (x-13)^6 (x^2-19x+74).
\end{align*}
The polynomial $\mathfrak{f}_1(x)$ is $\operatorname{Char}_S(x)$-warranted with certificate of warranty $$(45911387, 0, 0, 10146, 0)$$ and it has $208$ interlacing characteristic polynomials.
However, we arrive at a contradiction since $\mathfrak{f}_1(x)$ has a corresponding certificate of infeasibility:
$$
(1132367732240930, 0, 0, 71075114863, 6707563763, 649903361, 64522780, 6545789). \qedhere
$$
\end{proof}

\subsection{Seidel-compatible polynomials}

In this section, we introduce the notions of angles and Seidel-compatibility for polynomials.

    Let $\mathfrak f(x) = \dfrac{p(x)}{\operatorname{Min}_p(x)} f(x) \in \operatorname{Deck}(p)$.
    For each $\lambda \in \Lambda_p$, define the \textbf{angle} $\alpha_\lambda(\mathfrak f)$ of $\mathfrak f(x)$ with respect to $\lambda$ as
    \[
    \alpha_\lambda(\mathfrak f) \coloneqq \sqrt{\frac{f(\lambda)}{\operatorname{Min}_p^\prime(\lambda)}}.
    \]
    Now we can introduce the notion of \textit{compatibility} for polynomials with respect to $p(x)$.
    Let $\Sigma_p$ be the set of simple zeros of $p(x)$ and define $\operatorname{Sim}_p(x)$ as 
    \[
    \operatorname{Sim}_p(x) \coloneqq \prod_{\xi \in \Sigma_p} (x-\xi).
    \]
    
Let $\mathfrak f(x)$ and $\mathfrak g(x)$ be distinct polynomials in $\operatorname{Deck}(p)$.
Define $\operatorname{Mult}_p(x) \coloneqq \operatorname{Min}_p(x)/\operatorname{Sim}_p(x)$.
We say that $\mathfrak f(x)$ and $\mathfrak g(x)$ are $p(x)$-\textbf{Seidel-compatible} if there exists $\bm \delta \in \{\pm 1 \}^{\Sigma_p}$ such that 
    \begin{equation}
    \label{eqn:integralitySeidel}
    \sum_{\lambda \in \Sigma_p}  \operatorname{Mult}_p(\lambda)\bm \delta(\lambda) \alpha_{\lambda}(\mathfrak f)\alpha_{\lambda}(\mathfrak g)  \equiv R_p \pmod 2,
    \end{equation}
    where
    \[
    R_p \coloneqq 
    \begin{cases}
          \operatorname{Mult}_p(1) + \operatorname{Mult}_p(0), & \text{ if $\deg p$ is odd; } \\
         \left( \operatorname{Mult}_p(1) - \operatorname{Mult}_p(-1) \right)/2, & \text{ if $\deg p$ is even. }
    \end{cases}
    \]
Any polynomial $p(x)$ is considered to be $p(x)$-Seidel-compatible with itself.
Note that the definition of Seidel compatibility given here is a slight variation of the definition given in \cite{GSYdim1718}.

\begin{lem}\label{lem:compatible}
Let $S$ be a Seidel matrix of order $n$.
Suppose that $\mathfrak f(x) \in \operatorname{Deck}(\operatorname{Char}_S)$ is $\operatorname{Char}_S(x)$-warranted.
Then for all $j \in \{1,\dots,n\}$, the polynomial $\operatorname{Char}_{S[j]}(x)$ is $\operatorname{Char}_S(x)$-Seidel-compatible with $\mathfrak f(x)$.
\end{lem}
\begin{proof}
Set $p(x) = \operatorname{Char}_S(x)$.
Since $\mathfrak f(x) \in \operatorname{Deck}(p)$ is $p(x)$-warranted, then, by Lemma~\ref{lem:warranted}, there exists $i \in \{1,\dots,n\}$ such that $\operatorname{Char}_{S[i]}(x) = \mathfrak f(x)$.
Let $j \in \{1,\dots,n\}$ and let $\mathfrak{g}(x) = \operatorname{Char}_{S[j]}(x)$.
If $\mathfrak{f}(x) = \mathfrak{g}(x)$ then we are done.
Otherwise, we have $i \ne j$ and for each $\lambda \in \Sigma_p$, denote by $\mathbf u_\lambda$ a unit eigenvector for $\lambda$.
By the Spectral Decomposition Theorem, we have 
$$\operatorname{Mult}_p(S)_{i,j} =  \sum_{\lambda \in \Sigma_p} \operatorname{Mult}_p(\lambda)\mathbf u_\lambda(i) \mathbf u_\lambda(j) = \sum_{\lambda \in \Sigma_p}  \operatorname{Mult}_p(\lambda)\bm \delta(\lambda) \alpha_{\lambda}(\mathfrak f)\alpha_{\lambda}(\mathfrak g)$$
where $\bm \delta \in \{\pm 1 \}^{\Sigma_p}$ and $\mathbf{u}_{\lambda}(i)^2 = \alpha_{\lambda} (\mathfrak f)^2$, $\mathbf{u}_{\lambda}(j)^2 = \alpha_{\lambda} (\mathfrak g)^2$ by \cite[Lemmas 4.4 and 4.5]{GSYdim1718}.
Clearly, $\operatorname{Mult}_p(S)_{i,j}$ is an off-diagonal entries of $\operatorname{Mult}_p(S)$, which is an integer matrix.
Furthermore, the parities of the off-diagonal entries of $\operatorname{Mult}_p(S)$ can be determined by using \cite[Lemma 2.1]{GG18}.
It follows that $\operatorname{Mult}_p(S)_{i,j}$ also satisfies \eqref{eqn:integralitySeidel} and hence, $\mathfrak{g}(x)$ is $p(x)$-Seidel-compatible with $\mathfrak{f}(x)$.
\end{proof}

Now we employ Lemma~\ref{lem:compatible} to rule out the last remaining candidate characteristic polynomial from part (ii) of Theorem~\ref{thm:candpols}.

\begin{lem} \label{lem:last5_quad109}
There does not exist a Seidel matrix $S$ with characteristic polynomial
$$
\operatorname{Char}_S(x) = (x+5)^{42} (x-11)^{10} (x-13)^6 (x^2-22x+109).
$$
\end{lem}

\begin{proof}
Suppose a Seidel matrix $S$ has characteristic polynomial
$$
\operatorname{Char}_S(x) = (x+5)^{42} (x-11)^{10} (x-13)^6 (x^2-22x+109).
$$
Then $\operatorname{Deck}(\operatorname{Char}_S) = \{\mathfrak f_1(x), \dots, \mathfrak f_{11}(x)\}$, where
\begin{align*}
    \mathfrak{f}_1(x) &= (x+5)^{41} (x-11)^9 (x-13)^5 (x^4-41x^3+609x^2-3871x+8886), \\
    \mathfrak{f}_2(x) &= (x+5)^{41} (x-11)^9 (x-13)^5 (x^4-41x^3+609x^2-3867x+8834), \\
    \mathfrak{f}_3(x) &= (x+5)^{41} (x-11)^9 (x-13)^6 (x^3-28x^2+245x-682), \\
    \mathfrak{f}_4(x) &= (x+5)^{41} (x-11)^9 (x-13)^5 (x^4-41x^3+609x^2-3855x+8678), \\ 
    \mathfrak{f}_5(x) &= (x+5)^{41} (x-11)^9 (x-13)^6 (x^3-28x^2+245x-670), \\
    \mathfrak{f}_6(x) &= (x+5)^{41} (x-11)^9 (x-13)^5 (x^4-41x^3+609x^2-3851x+8626), \\
    \mathfrak{f}_7(x) &= (x+5)^{41} (x-9) (x-11)^9 (x-13)^6 (x^2-19x+74), \\
    \mathfrak{f}_8(x) &= (x+5)^{41} (x-5) (x-11)^{11} (x-13)^5 (x-14), \\
    \mathfrak{f}_9(x) &= (x+5)^{41} (x-11)^9 (x-13)^6 (x^3-28x^2+245x-654), \\
    \mathfrak{f}_{10}(x) &= (x+5)^{41} (x-5) (x-10) (x-11)^9 (x-13)^7, \\
    \mathfrak{f}_{11}(x) &= (x+5)^{41} (x-11)^{10} (x-13)^6 (x^2-17x+58).
\end{align*}
The polynomial $\mathfrak{f}_1(x)$ is $\operatorname{Char}_S(x)$-warranted with certificate of warranty $$(16485427, 0, 0, 0, -1859).$$
Only three polynomials, $\mathfrak{f}_1(x), \mathfrak{f}_3(x),$ and $\mathfrak{f}_8(x)$, in $\operatorname{Deck}(\operatorname{Char}_S)$ are $\operatorname{Char}_S(x)$-Seidel-compatible with $\mathfrak f_1(x)$.
The vector $\mathbf n = (207/4,6,9/4)$ is the unique solution to the equation $\mathbf{n} \cdot \operatorname{Coeff}(\mathfrak{f}_1,\mathfrak{f}_3,\mathfrak{f}_8) = \operatorname{Coeff}(\operatorname{Char}_S^\prime)$.
However, the entries of $\mathbf n$ are not all integers, which contradicts Lemma~\ref{lem:interlacingconfig}.
\end{proof}
To complete the proof of Theorem~\ref{thm:main}, it remains to rule out one last candidate characteristic polynomial from Theorem~\ref{thm:candpols}.

\section{The Jacobi identity for complementary subgraphs}
\label{sec:decaen_jacobi}

In this section, we rule out the last remaining candidate characteristic polynomial from Theorem~\ref{thm:candpols}.
That is, we prove the following theorem.

\begin{thm} \label{thm:last5_3ev}
There does not exist a Seidel matrix $S$ with characteristic polynomial
$$
\operatorname{Char}_S(x) = (x+5)^{42} (x-11)^{15} (x-15)^3.
$$
\end{thm}

Let $\Gamma$ be a graph, $A$ be the adjacency matrix of $\Gamma$, and $\mathcal T, \mathcal U$ be subsets of $V(\Gamma)$.
In the remainder of the paper, the rows and columns of the adjacency matrix $A$ are indexed by the vertices of $\Gamma$.
Accordingly, we define the following notation.
\begin{itemize}
    \item $\Gamma[\mathcal T]$ denotes the subgraph of $\Gamma$ induced on the vertices of $\mathcal T$.
    \item $A[\mathcal T, \mathcal U]$ denotes the submatrix of $A$ with rows from $\mathcal T$ and columns from $\mathcal U$.
    \item $A[\mathcal T]$ denotes the principal submatrix $A[\mathcal T,\mathcal T]$.
\end{itemize}

Let $v$ be a vertex of $\Gamma$ and let $\mathbf{v}$ be a vector such that $\mathbf{v} = A[\mathcal{T},\{v\}]$.
For $\mathcal S \subseteq \mathcal T$, denote by $\mathbf{v}[\mathcal{S}]$ the vector $A[\mathcal{S},\{v\}]$.
Furthermore, in this section we occasionally consider the complement of a subset of $V(\Gamma)$, where similarly to before, we define $\overline{\mathcal T} \coloneqq V(\Gamma) \backslash \mathcal{T}$.
In order to prove Theorem~\ref{thm:last5_3ev}, inspired by a paper of De~Caen~\cite{decaen}, we heavily use the following identity due to Jacobi:
\begin{lem}[Jacobi]
\label{lem:decaenjacobi}
    Let $\mathcal T \subseteq V(\Gamma)$ and suppose that $A$ is the adjacency matrix of $\Gamma$, then
    \[
    \det(xI - A[\overline{\mathcal T}]) = \det(xI-A) \det( (xI-A)^{-1}[\mathcal T] ).
    \]
\end{lem}

We use Lemma~\ref{lem:decaenjacobi} to derive a necessary algebraic condition for induced subgraphs of $\Gamma$ (see Corollary~\ref{cor:decaen} and Lemma~\ref{lem:decaenBlock}). 

Denote by $J$ the all-ones matrix and by $O$ the zero matrix.
We use subscripts $J_{n,m}$ to indicate the all-ones matrix with $n$ rows and $m$ columns, or merely write $J_n$ for $J_{n,n}$.
The all-ones vector (with $n$ entries) is denoted by $\mathbf 1_n$.
The subscript is omitted if the order of the matrix (or vector) can be determined from context. 

First we consider the orthogonal projection matrix of the dimension-$2$ eigenspace of a regular graph in the switching class of the putative Seidel matrix of Theorem~\ref{thm:last5_3ev}.
For the sake of contradiction, suppose that there exists a Seidel matrix $S$ with characteristic polynomial
$$
\operatorname{Char}_S(x) = (x+5)^{42} (x-11)^{15} (x-15)^3.
$$
By \cite[Remark 5.12]{GG18} there exists a (regular) graph $\Gamma$ with adjacency matrix $A$ of spectrum:
\[\Spec A=\{[22]^1,[2]^{42},[-6]^{15},[-8]^2\}.\]
For the rest of this section, $\Gamma$ will remain fixed and will be repeatedly referred to without being specified again.

\subsection{An equitable partition for $\Gamma$}

Following \cite[Section~3.1]{vDS}, we set $M=3A^2+12A-36I-28J$, where $A$ is the adjacency matrix of $\Gamma$.
Then $M$ is positive semidefinite with diagonal entries $2$,
so $|M_{ij}|\leqslant 2$ for all $i,j$. 
For $i\neq j$, the entry $(A^2+4A)_{ij} = (M_{ij}+28)/3$ is an integer.
Hence
$M_{ij} \in \{-1,2\}$ and $(A^2+4A)_{ij} \in \{9,10\}$.

Now we claim that the relation where $i\sim j$ when $M_{ij}=2$
is an equivalence relation. 
Indeed, if
$i\sim j$ and $j\sim k$ for distinct $i,j,k$, then
\begin{align*}
2 &=\rank M \geqslant \rank \begin{bmatrix}
M_{ii}&M_{ij}&M_{ik}\\
M_{ij}&M_{jj}&M_{jk}\\
M_{ik}&M_{jk}&M_{kk}
\end{bmatrix}
= \rank\begin{bmatrix}
2&2&M_{ik}\\
2&2&2\\
M_{ik}&2&2\end{bmatrix}.
\end{align*}
This forces $M_{ik}=2$, or equivalently, $i\sim k$.
Now, since $MJ=0$ and the entries of $M$ are $2$ and $-1$,
each equivalence class has size $20$. 

Let $\Omega$ be the partition of $V(\Gamma)$ whose three parts are the three equivalence classes of the above relation.
Just as for $\Gamma$, we also fix $\Omega$ for the remainder of this section.
If we assume that $ A = \left [\begin{smallmatrix}
A_{11}&A_{12}&A_{13}\\
A_{21}&A_{22}&A_{23}\\
A_{31}&A_{32}&A_{33}\end{smallmatrix} \right ]$ where $\{ A_{ii} \; : \; i \in 
\{1,2,3\} \}= \{ A[\mathcal T] \; : \; \mathcal T \in \Omega \}$ then
\[
M=\begin{bmatrix}
2&-1&-1\\
-1&2&-1\\
-1&-1&2
\end{bmatrix}\otimes J_{20},
\]
where $\otimes$ denotes the Kronecker product for matrices.

We summarise the above as a lemma:

\begin{lem}
\label{lem:almostsrg}
Suppose $A$ is the adjacency matrix of $\Gamma$ such that $ A = \left [\begin{smallmatrix}
A_{11}&A_{12}&A_{13}\\
A_{21}&A_{22}&A_{23}\\
A_{31}&A_{32}&A_{33}\end{smallmatrix} \right ]$ where $\{ A_{ii} \; : \; i \in 
\{1,2,3\} \}= \{ A[\mathcal T] \; : \; \mathcal T \in \Omega \}$. Then
    \begin{align*}
    A^2+4A-12I_{60} = 9J_{60}+I_3 \otimes J_{20}.
\end{align*}
\end{lem}

Now we show that $\Omega$ is an equitable partition of $\Gamma$.

\begin{lem}
\label{lem:equiPart}
     The partition $\Omega$ of $V(\Gamma)$ is equitable with quotient matrix 
     \[
     \begin{bmatrix}
     2 & 10 & 10 \\
     10 & 2 & 10 \\
     10 & 10 & 2
     \end{bmatrix}.
     \]
\end{lem}
\begin{proof}
Suppose that $\Omega = \{\mathcal T_1,\mathcal T_2,\mathcal T_3\}$.
Assume that $A$ is the adjacency matrix of $\Gamma$ such that $ A = \left [\begin{smallmatrix}
A_{11}&A_{12}&A_{13}\\
A_{21}&A_{22}&A_{23}\\
A_{31}&A_{32}&A_{33}\end{smallmatrix} \right ]$ where $A_{ii} = A[\mathcal T_i]$ for each $i \in 
\{1,2,3\}$.

Recall that $M=3A^2+12A-36I-28J$.
Using $(A+8I)(A+6I)(A-2I)=280J$ and Lemma~\ref{lem:almostsrg}, we deduce that
\begin{align}
\label{eqn:mateq}
        AM &= \begin{bmatrix}
        -16&8&8\\
        8&-16&8\\
        8&8&-16
\end{bmatrix}\otimes J_{20}.
\end{align}
Let $u_i \in \mathcal T_i$ for $i = 1,2,3$ and denote by $n_{ij}$ the number of neighbours of $u_i$ in $\mathcal T_j$.
Take the $3 \times 3$ submatrix of \eqref{eqn:mateq} induced on $\{u_1,u_2,u_3\}$ to obtain
\begin{align*}
\begin{bmatrix}
n_{11} & n_{12} & n_{13} \\
n_{21} &n_{22} &n_{23}\\
n_{31} & n_{32} & n_{33} \end{bmatrix}
\begin{bmatrix}
2&-1&-1\\
-1&2&-1\\
-1&-1&2\end{bmatrix} &=
 \begin{bmatrix}
-16&8&8\\
8&-16&8\\
8&8&-16\end{bmatrix}.
\end{align*}
Since $\Gamma$ is regular with valency $22$, we obtain
\begin{align*}
\begin{bmatrix}
n_{11} & n_{12} & n_{13} \\
n_{21} &n_{22} &n_{23}\\
n_{31} & n_{32} & n_{33} \end{bmatrix}
 &=
 \begin{bmatrix}
2&10&10\\
10&2&10\\
10&10&2\end{bmatrix}.
\end{align*}
The statement of the lemma follows, since the choice of $u_1$, $u_2$, and $u_3$ was immaterial.
\end{proof}

The following corollary is immediate.

\begin{cor}
\label{cor:part210}
Let $\mathcal T \in \Omega$ and $u \in \overline{\mathcal T}$.
Then $\Gamma[\mathcal T]$ is a disjoint union of cycles and $u$ is adjacent to precisely $10$ vertices in $\mathcal T$.
\end{cor}

\subsection{Expressions for matrix inverses }

Now we pursue expressions for the inverse of matrices that we may encounter when applying Lemma~\ref{lem:decaenjacobi}.
Let $D:\mathbb R[x] \to \mathbb R[x]$ be the linear operator defined by $Dx^i \coloneqq x^{i-1}$ for $i>0$ and $Dx^0 \coloneqq 0$.

\begin{lem} \label{lem:meanval_cayleyhamilton}
Suppose $A$ is a real square matrix and $m(x)$ denotes its minimal polynomial.
Let $f$ be a rational function over the real numbers.
If $m(f(x))$ is nonzero, then we have
\[
(f(x) I-A)^{-1} = \frac{1}{m(f(x))} \sum_{k=1}^{\deg m} D^km(f(x)) A^{k-1}.
\]
\end{lem}
\begin{proof}
Observe that
\begin{align}
    m(y) &= \sum_{k=0}^{\deg m} \left(D^km(t)-tD^{k+1}m(t)\right) y^k \nonumber \\ 
    m(y)-m(t) &= (y-t) \sum_{k=1}^{\deg m} D^km(t) \cdot y^{k-1}. \label{eq:meanval}
\end{align}
Suppose we set $y=A$ and $t=f(x)$.
By the Cayley-Hamilton theorem, $m(A)=O$ and hence,
\begin{align*}
-m(f(x))I = (A-f(x)I)\sum_{k=1}^{\deg m} D^k m(f(x)) A^{k-1}.
\end{align*}
The desired result then follows immediately.
\end{proof}

Next we derive an expression for $(xI-A)^{-1}$, where $A$ is the adjacency matrix of $\Gamma$.

\begin{prop}
\label{prop:cinv}
Suppose $A$ is the adjacency matrix of $\Gamma$ such that $ A = \left [\begin{smallmatrix}
A_{11}&A_{12}&A_{13}\\
A_{21}&A_{22}&A_{23}\\
A_{31}&A_{32}&A_{33}\end{smallmatrix} \right ]$ where $\{ A_{ii} \; : \; i \in 
\{1,2,3\} \}= \{ A[\mathcal T] \; : \; \mathcal T \in \Omega \}$.
Then
\begin{align*}
     (xI-A)^{-1} = \frac{(x-22)(x+8)\left(A + (x+4)I_{60}\right) + (9x+82)J_{60} + (x-22)I_3\otimes J_{20}}{(x-22)(x-2)(x+6)(x+8)}.
\end{align*}
\end{prop}

\begin{proof}
Let $m(x) = (x-22)(x-2)(x+6)(x+8) =x^4-10x^3-244x^2-536x+2112$ be the minimal polynomial of $A$.
We obtain
\begin{align*}
    D^1m(x) &= x^3-10x^2-244x-536, \\
    D^2m(x) &= x^2-10x-244, \\
    D^3m(x) &= x-10, \\
    D^4m(x) &= 1.
\end{align*}
By Lemma~\ref{lem:meanval_cayleyhamilton}, we obtain
\begin{align}
    \label{eqn:mexp}
    m(x)(xI-A)^{-1} 
    = A^3 + (x-10) A^2 + (x^2-10x-244)A + (x^3-10x^2-244x-536)I.
\end{align}
By Lemma~\ref{lem:almostsrg},
\begin{align*}
    A^3 = -4A^2 + 12A + 208J_{60} - 8I_3 \otimes J_{20} = 28A - 48I_{60} + 172J_{60} -12I_3 \otimes J_{20}.
\end{align*}
Furthermore,
\(
    A (I_3 \otimes J_{20}) = (I_3 \otimes J_{20}) A = 10J_{60} - 8I_3 \otimes J_{20}.
\)
The lemma then follows by simplifying \eqref{eqn:mexp}.
\end{proof}

Let $\oplus$ denote the direct sum of matrices.
Applying Proposition~\ref{prop:cinv} and Lemma~\ref{lem:decaenjacobi} together with the fact that the characteristic polynomial of a graph belongs to $\mathbb{Z}[x]$, we obtain the following corollary:

\begin{cor}[Subgraph condition I]
\label{cor:decaen}
Let $\mathcal{T}, \mathcal U \in \Omega$ such that $\mathcal T \ne \mathcal U$.
    Suppose that $X = \left [\begin{smallmatrix}
X_{11}&X_{12}\\
X_{21}&X_{22}\end{smallmatrix} \right ]$ is an adjacency matrix for $\Gamma[\mathcal T \cup \mathcal S]$ where $\mathcal S \subseteq \mathcal U$, $X_{11} = X[\mathcal T]$ and $X_{22} = X[\mathcal S]$.
    Then the determinant of 
    \begin{equation*}
    \label{eqn:decaen}
     (x+4)I_{20+|\mathcal S|} + X + \frac{(9x+82)J_{20+|\mathcal S|}}{(x-22)(x+8)}+ \frac{J_{20}\oplus J_{|\mathcal S|}}{x+8}
    \end{equation*}
    is in $\displaystyle \frac{(x+6)^{5+|\mathcal S|}}{(x-22)(x-2)^{22-|\mathcal S|}(x+8)^2} \mathbb Z[x]$.
\end{cor}
\begin{proof}
Suppose $A$ is the adjacency matrix of $\Gamma$ such that $A[\mathcal{T} \cup \mathcal{S}] = X$.
By Proposition~\ref{prop:cinv},
\[
(x-2)(x+6)(xI-A)^{-1}[\mathcal{T} \cup \mathcal{S}] = (x+4)I_{20+|\mathcal S|} + X + \frac{(9x+82)J_{20+|\mathcal S|}}{(x-22)(x+8)}+ \frac{J_{20}\oplus J_{|\mathcal S|}}{x+8}.
\]
Since $\det(xI-A[\overline{\mathcal{T} \cup \mathcal{S}}]) \in \mathbb{Z}[x]$, the conclusion then follows from Lemma~\ref{lem:decaenjacobi}.
\end{proof}

\subsection{Five $4$-cycles}

In this subsection, we show that for each $\mathcal T \in \Omega$, the local graph $\Gamma[\mathcal T]$ is the disjoint union of five $4$-cycles.
Denote by $C_n$, the cycle graph on $n$ vertices.

\begin{lem}
\label{lem:40failfrom49}
Let $\mathcal T \in \Omega$.  Then $\Gamma[\mathcal T]$ is the disjoint union of at least $5$ cycles.
\end{lem}
\begin{proof}
Let $A$ be the adjacency matrix of $\Gamma$.
By Corollary~\ref{cor:decaen}, we have that
\[
\det \left (  (x+4)I + A[\mathcal T] + \frac{10(x+6)J}{(x-22)(x+8)} \right )
\]
must be in $\displaystyle \frac{(x+6)^5}{(x-22)(x-2)^{22}(x+8)^2} \mathbb Z[x]$.

By Corollary~\ref{cor:part210}, we know that $\Gamma[\mathcal T]$ is a disjoint union of cycles.
Let $\Psi$ be the multiset of eigenvalues (with multiplicity) of $A[\mathcal T]$ whose eigenspaces are orthogonal to the all-ones vector and let $c$ be the number of connected components of $\Gamma[\mathcal T]$.
By the matrix determinant lemma, we have
\[
\det \left (  (x+4)I + A[\mathcal T] + \frac{10(x+6)J}{(x-22)(x+8)} \right ) = \frac{(x+6)^c(x-12)(x-2)}{(x-22)(x+8)}\prod_{\lambda \in \Psi}(x+4+\lambda).
\]
The lemma follows since, by the Perron-Frobenius theorem~\cite[Chapter 8]{GoRo}, we have $|\lambda| < 2$ for each $\lambda \in \Psi$.
\end{proof}

Now we establish the main result of this subsection.

\begin{lem}
\label{lem:5squares}
Let $\mathcal T \in \Omega$.
Then $\Gamma[\mathcal T]$ is the disjoint union of five $4$-cycles.
\end{lem}

\begin{proof}
Suppose that $\Gamma[\mathcal T]$ is a graph of order $20$ equals to a disjoint union of cycles.
We will show that each cycle of $\Gamma[\mathcal T]$ must be an even cycle.
Let $u \in \overline{\mathcal T}$ and let $A$ be the adjacency matrix of $\Gamma$.
Applying Proposition~\ref{prop:cinv}, we write
    \begin{align*}
        (x-2)(x+6)(xI-A)^{-1}[\mathcal T \cup \{u\}] = \begin{pmatrix}
        P & Q \\
        Q^{\transpose} & R
        \end{pmatrix}
    \end{align*}
where
\begin{align*}
    P &= (x+4)I + \frac{10(x+6)J}{(x-22)(x+8)} + A[\mathcal T], \\
    Q &= \mathbf{u}+\frac{9x+82}{(x-22)(x+8)}\mathbf{1}, \\
    R &= (x+4)+\frac{10(x+6)}{(x-22)(x+8)} = \frac{x^3-10x^2-222x-644}{(x-22)(x+8)},
\end{align*}
and $\mathbf{u} = A[\mathcal{T},\{u\}]$.
Using the properties of Schur complement \cite[Theorem 3.1.1]{schurcomplement}, we can write
    \begin{align*}
        \det \left( (xI-A)^{-1}[\mathcal T\cup \{u\}] \right) &= \frac{\det P \cdot \det(R-Q^{\transpose} P^{-1} Q)}{(x-2)^{21} (x+6)^{21}}.
    \end{align*}
From Lemma~\ref{lem:decaenjacobi}, we obtain
\begin{align*}
    \det(xI - A[\overline {\mathcal T\cup \{u\}}]) &= \frac{(x-22) (x-2)^{21} (x+8)^2}{(x+6)^6} \cdot \det P \cdot \det(R-Q^{\transpose} P^{-1} Q).
\end{align*}
As in the proof of Lemma~\ref{lem:40failfrom49}, note that the multiplicity of $(x+6)$ as an eigenvalue of $P$ is equal to the number of disjoint cycles of $\Gamma[\mathcal T]$ and this number is at most $6$.
Therefore,
\[
\frac{(x-22) (x-2)^{21} (x+8)^2}{(x+6)^6} \cdot \det P
\]
in its simplest form cannot have $(x+6)$ as a factor of the numerator.

Applying the Sherman-Morrison formula \cite{shermanmorrison}, we have
\begin{align*}
    P^{-1} = ((x+4)I+A[\mathcal T])^{-1} - \frac{10 J}{(x-12)(x-2)(x+6)}.
\end{align*}
Note that
\begin{align*}
    \det(R-Q^{\transpose} P^{-1} Q) &=   \frac{x^4+6x^3-102x^2-1016x-2604}{(x-12)(x+6)(x+8)}
-\mathbf{u}^\transpose  ((x+4)I+A[\mathcal T])^{-1} \mathbf{u}.
\end{align*}
Now suppose that $\mathcal T= \bigcup_{i=1}^t \mathcal C_i$ where $\Gamma[\mathcal C_i] = C_{|\mathcal C_i|}$ for each $i \in \{1,\dots,t\}$.
Fix $i \in \{1,\dots,t\}$ and by Lemma~\ref{lem:meanval_cayleyhamilton}, we obtain
\begin{align*}
((x+4)I+A[\mathcal C_i])^{-1} = -\frac{1}{m_i(-x-4)} \sum_{k=1}^{\deg m_i} D^k m_i(-x-4) A[\mathcal C_i]^{k-1}
\end{align*}
where $m_i(x) = (x-2) q_i(x)$ is the minimal polynomial of $C_{|\mathcal C_{i}|}$.
Hence,
\begin{align}
\label{eq:tau}
(x+6)((x+4)I+A[\mathcal C_i])^{-1} = \frac{1}{q_i(-x-4)} \sum_{k=1}^{\deg m_i} D^k m_i(-x-4) A[\mathcal C_i]^{k-1}
\end{align}
Define $\omega(x) \coloneqq (x+6) \det(R-Q^{\transpose} P^{-1} Q)$ and $\tau_i(x) \coloneqq (x+6) ((x+4)I+A[\mathcal C_i])^{-1}$. 
Then
\begin{align*}
   \omega(x) &= \frac{x^4+6x^3-102x^2-1016x-2604}{(x-12)(x+8)} -   \sum_{i=1}^t \mathbf{u}[\mathcal C_i]^\transpose\tau_i(x) \mathbf{u}[\mathcal C_i].
\end{align*}
Note that $\det(xI - A[\overline {\mathcal T\cup \{u\}}]) \in \mathbb{Z}[x]$.
Thus, the denominator of $\det(R-Q^{\transpose} P^{-1} Q)$ in its simplest form is not divisible by $(x+6)$, since the other factor of $\det(xI - A[\overline {\mathcal T\cup \{u\}}])$ does not have $(x+6)$ as a factor in its numerator, as argued above.
Hence, we have $\omega(-6) = 0$.
Next, we will show that $\omega(-6) = 0$ implies that $|\mathcal C_i|$ is even for all $i$.
By \eqref{eq:tau}, we have
\begin{align*}
    \tau_i(-6) &= \frac{1}{q_i(2)} \sum_{k=1}^{\deg m_i} D^k m_i(2) A[\mathcal C_i]^{k-1}.
\end{align*}
Using \eqref{eq:meanval}, we obtain
\begin{align*}
    (A[\mathcal C_i]-2I)\sum_{k=1}^{\deg m_i} D^k m_i(2) A[\mathcal C_i]^{k-1} = m_i(A[\mathcal C_i])-m_i(2)I = 0,
\end{align*}
which implies that
\[
    \sum_{k=1}^{\deg m_i} D^k m_i(2) A[\mathcal C_i]^{k-1} = \varepsilon J_{|\mathcal C_i|}
\]
for some constant $\varepsilon$.
Multiplying both sides from the right by the all-ones vector yields
\[
    \varepsilon = \frac{1}{|\mathcal C_i|} \sum_{k=1}^{\deg m_i} D^k m_i(2) \cdot 2^{k-1} = \frac{m_i'(2)}{|\mathcal C_i|} = \frac{q_i(2)}{|\mathcal C_i|}.
\]
It follows that
\begin{align*}
    \tau_i(-6) = \frac{1}{q_i(2)} \cdot \frac{q_i(2)}{|\mathcal C_i|} J_{|\mathcal C_i|} = \frac{J_{|\mathcal C_i|}}{|\mathcal C_i|}
\end{align*}
and therefore, 
\[
   \omega(-6) = 5 -  \sum_{i=1}^t \frac{(\mathbf 1^\transpose \mathbf{u}[\mathcal C_i])^2}{|\mathcal C_i|}.
\]
By the Cauchy-Schwarz inequality, we have
\begin{align*}
    \sum_{i=1}^t \frac{(\mathbf 1^\transpose \mathbf{u}[\mathcal C_i])^2}{|\mathcal C_i|} \geqslant \frac{(\mathbf{1}^{\transpose} \mathbf{u})^2}{\sum_{i=1}^t |\mathcal C_i|} = \frac{100}{20} = 5
\end{align*}
with equality if and only if there exists $\delta$ such that $\mathbf{1}^{\transpose} \mathbf{u}[\mathcal C_i] = \delta |\mathcal C_i|$ for all $i$.
Since $\omega(-6) = 0$, there exists $\delta$ such that $\mathbf{1}^{\transpose} \mathbf{u}[\mathcal C_i] = \delta |\mathcal C_i|$ for all $i$.
Moreover, we have $\delta = 1/2$ since $20\delta = \mathbf{1}^{\transpose} \mathbf{u} = 10$.
Since $\mathbf{1}^{\transpose} \mathbf{u}[\mathcal C_i] = |\mathcal C_i|/2$ is an integer, each cycle of $\Gamma[\mathcal T]$ must be even.
Combined with Lemma~\ref{lem:40failfrom49}, we conclude that $\Gamma[\mathcal T]$ is the disjoint union of five 4-cycles.
\end{proof}

\subsection{Pairwise compatible neighbourhoods}

Let $\mathcal T \in \Omega$.
In this subsection, we consider graphs resulting from joining an independent set of vertices from some $\mathcal U \in \Omega$ where $\mathcal U \ne \mathcal T$ to the vertex set of $\Gamma[\mathcal T]=5C_4$.
The next lemma provides a necessary algebraic condition for such graph to be an induced subgraph of $\Gamma$.

\begin{lem}[Subgraph condition II]
\label{lem:decaenBlock}
Let $\mathcal{T}, \mathcal U \in \Omega$ such that $\mathcal T \ne \mathcal U$.
    Suppose $\mathcal S \subseteq \mathcal U$ such that $|\mathcal S| = s$ and the vertices in $\mathcal S$ are pairwise nonadjacent.
    Let $A$ be the adjacency matrix of $\Gamma$.
    Suppose $A[\mathcal T,\mathcal T\cup \mathcal S] = \begin{bmatrix}X&B\end{bmatrix}$, where $X = A[\mathcal T]$.
    Then the determinant of
    \[
    (x+4)I+\frac{10(x^2 + 4x-30)J}{(x-12)(x+6)(x+8)}-\frac{B^\transpose B}{x+4}-\frac{B^\transpose (X^2-(x+4)X) B}{(x+2)(x+4)(x+6)}
    \]
    is in $\displaystyle \frac{(x+6)^s}{(x-12)(x-2)^{23-s}(x+2)^5(x+4)^{10}(x+8)} \mathbb Z[x]$.
\end{lem}
\begin{proof}
    Using properties of the Schur complement, we can write
    \begin{align}
    \label{eqn:cor46}
            \det \left (  (x+4)I+\begin{bmatrix}
     X & B \\
     B^\transpose & O
     \end{bmatrix}+\frac{(9x+82)J}{(x-22)(x+8)}+ \frac{J_{20}\oplus J_{s}}{x+8} \right ) &= \det P \cdot \det(R-Q^\transpose P^{-1} Q),
    \end{align}
    where
    \begin{align*}
    P &= (x+4)I_{20} + \frac{10(x+6)}{(x-22)(x+8)}J_{20} + X;\\
    Q &= B + \frac{9x+82}{(x-22)(x+8)}J_{20,s};\\
    R &= (x+4)I_{s} + \frac{10(x+6)}{(x-22)(x+8)}J_{s}.
\end{align*}
By Lemma~\ref{lem:5squares}, the induced subgraph $\Gamma[\mathcal T] = 5C_4$.
Thus, we can compute $\det P$ and $P^{-1}$ as follows:
\begin{align}\label{eqn:det}
    \det P &= \frac{(x-12)(x-2)(x+2)^5(x+4)^{10}(x+6)^5}{(x-22)(x+8)}; \\
    \nonumber
    P^{-1} &= \frac{I_{20}}{x+4} - \frac{10J_{20}}{(x-12)(x-2)(x+6)} + \frac{X^2-(x+4)X}{(x+2)(x+4)(x+6)}.
\end{align}
By Lemma~\ref{lem:equiPart}, we know that $B^{\transpose} \mathbf 1_{20} = 10 \mathbf 1_{s}$.
Hence
\begin{align}
\label{eqn:rqpq}
    R-Q^\transpose P^{-1} Q =   (x+4)I_s+\frac{10(x^2 + 4x-30)J_s}{(x-12)(x+6)(x+8)}-\frac{B^\transpose B}{x+4}-\frac{B^\transpose (X^2-(x+4)X) B}{(x+2)(x+4)(x+6)}.
\end{align}
The lemma then follows from Corollary~\ref{cor:decaen} with \eqref{eqn:cor46}, \eqref{eqn:det}, and \eqref{eqn:rqpq}.
\end{proof}

Let $u \in \overline{\mathcal T}$.
First we consider graphs resulting from joining $u$ to $\mathcal{T}$.
Based on Lemma~\ref{lem:decaenBlock} above, the following lemma yields conditions on the adjacency of $u$ with the vertices in $\mathcal{T}$.

\begin{lem}
\label{lem:21}
    Let $\mathcal T\in \Omega$ and $u \in \overline{\mathcal T}$.
    Let $A$ be the adjacency matrix of $\Gamma$.
    Suppose that $\mathbf u = A[\mathcal T, \{u\}]$.
    Then
        \[
        \mathbf u^\transpose \mathbf u = 10, \quad \mathbf u^\transpose A[\mathcal T]\mathbf u = 6, \quad \text{ and } \quad  \mathbf u^\transpose A[\mathcal T]^2\mathbf u = 28.
        \]
\end{lem}
\begin{proof}
By Lemma~\ref{lem:equiPart}, the vertex $u$ is adjacent to $10$ vertices in $\Gamma[\mathcal T]$.
Thus, $\mathbf 1^\transpose \mathbf u = 10$.
By Lemma~\ref{lem:decaenBlock},
\begin{align*}
        & (x+4)+\frac{10(x^2 + 4x-30)}{(x-12)(x+6)(x+8)}-\frac{10}{x+4}-\frac{\mathbf u^\transpose A[\mathcal T]^2\mathbf u}{(x+2)(x+4)(x+6)}+\frac{\mathbf u^\transpose A[\mathcal T]\mathbf u}{(x+2)(x+6)}
\end{align*}
belongs to $\displaystyle \frac{(x+6)}{(x-12)(x-2)^{22}(x+2)^5(x+4)^{10}(x+8)}\mathbb Z[x]$.
Set $\alpha = \mathbf u^\transpose A[\mathcal T]^2 \mathbf u$ and $\beta = \mathbf u^\transpose  A[\mathcal T] \mathbf u$.
It follows that 
\[
(x+2)(x^5+10x^4-88x^3-1444x^2-5468x-4656)  -(x-12)(x+8) \left(\alpha - (x+4)\beta \right)
\]
is in $(x+6)^2\mathbb Z[x]$.
Reducing modulo $(x+6)^2$ yields the simultaneous equations
$16\alpha - 4 \beta = 424$ and $132\alpha +48\beta = 3984$.
Thus $\alpha = 28$ and $\beta = 6$.
\end{proof}

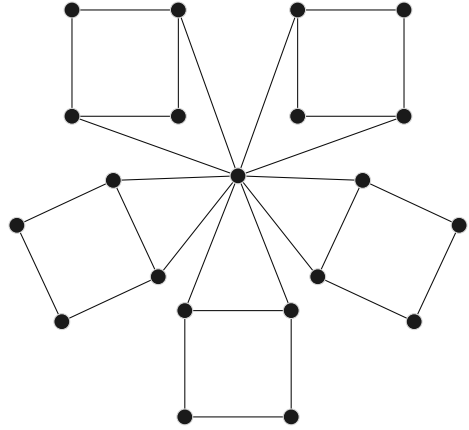
\begin{figure}[htbp]
	\centering
		\begin{tikzpicture}[rotate=90]
			\tikzstyle{vertex}=[circle,thin,draw=black!15,fill=black!90, inner sep=0pt, minimum width=6pt]
			\tikzstyle{edge}=[draw=black!95,-]
			\tikzstyle{fedge}=[draw=black!75,-]
	
			\begin{scope}[xshift=0cm]
				\node[vertex] (center) at (0:0) {};
			\end{scope}
			\begin{scope}[xshift=-2.5cm]
				\newdimen\rad
				\rad=1cm
				\def\angle{90}
				\def\offsetI{-45}
				\foreach \x in {0,...,3}
				{
					\draw[edge] (\offsetI+\angle*\x:\rad) -- (\offsetI+\angle+\angle*\x:\rad);
			    }
			    \foreach \x in {0,...,1}
				{
					\draw[edge] (\offsetI+\angle*\x:\rad) -- (center);
			    }
				\foreach \x in {0,...,3}
				{
					\draw (\offsetI+\angle*\x:\rad) node[vertex] {};
			    }
			\end{scope}
						\begin{scope}[xshift=-1cm, yshift=2cm]
				\newdimen\rad
				\rad=1cm
				\def\angle{90}
				\def\offsetI{-20}
				\def\offsetII{-18}
				\foreach \x in {0,...,3}
				{
					\draw[edge] (\offsetI+\angle*\x:\rad) -- (\offsetI+\angle+\angle*\x:\rad);
			    }
			    \foreach \x in {3,0}
				{
					\draw[edge] (\offsetI+\angle*\x:\rad) -- (center);
			    }
				\foreach \x in {0,...,3}
				{
					\draw (\offsetI+\angle*\x:\rad) node[vertex] {};
			    }
			\end{scope}
			\begin{scope}[xshift=-1cm, yshift=-2cm]
				\newdimen\rad
				\rad=1cm
				\def\angle{90}
				\def\offsetI{-70}
				\def\offsetII{-18}
				\foreach \x in {0,...,3}
				{
					\draw[edge] (\offsetI+\angle*\x:\rad) -- (\offsetI+\angle+\angle*\x:\rad);
			    }
			    \foreach \x in {1,2}
				{
					\draw[edge] (\offsetI+\angle*\x:\rad) -- (center);
			    }
				\foreach \x in {0,...,3}
				{
					\draw (\offsetI+\angle*\x:\rad) node[vertex] {};
			    }
			\end{scope}
			\begin{scope}[xshift=1.5cm, yshift=-1.5cm]
				\newdimen\rad
				\rad=1cm
				\def\angle{90}
				\def\offsetI{-45}
				\def\offsetII{-18}
				\foreach \x in {0,...,3}
				{
					\draw[edge] (\offsetI+\angle*\x:\rad) -- (\offsetI+\angle+\angle*\x:\rad);
			    }
			    \foreach \x in {3,1}
				{
					\draw[edge] (\offsetI+\angle*\x:\rad) -- (center);
			    }
				\foreach \x in {0,...,3}
				{
					\draw (\offsetI+\angle*\x:\rad) node[vertex] {};
			    }
			\end{scope}
						\begin{scope}[xshift=1.5cm, yshift=1.5cm]
				\newdimen\rad
				\rad=1cm
				\def\angle{90}
				\def\offsetI{45}
				\def\offsetII{-18}
				\foreach \x in {0,...,3}
				{
					\draw[edge] (\offsetI+\angle*\x:\rad) -- (\offsetI+\angle+\angle*\x:\rad);
			    }
			    \foreach \x in {3,1}
				{
					\draw[edge] (\offsetI+\angle*\x:\rad) -- (center);
			    }
				\foreach \x in {0,...,3}
				{
					\draw (\offsetI+\angle*\x:\rad) node[vertex] {};
			    }
			\end{scope}
		\end{tikzpicture}
	\caption{The 21-vertex graph $\mathfrak G$.}
	\label{fig:graphG}
\end{figure}

By Corollary~\ref{cor:part210}, we know that $u$ has $10$ neighbours in $\mathcal T$.
At first sight, there appears to be $\binom{20}{10}$ possible ways to assign vertices in $\mathcal T$ as neighbours of $u$.
However, using Lemma~\ref{lem:21}, we will see that only 2560 of these $10$-sets of $\mathcal T$ satisfy Corollary~\ref{cor:21} below.
Furthermore, all 2560 such graphs are isomorphic to the graph $\mathfrak{G}$ in Figure~\ref{fig:graphG}.

For each $u \in V(\Gamma)$, we denote by $N(u)$ the subset of $V(\Gamma)$ consisting of vertices adjacent to $u$.
\begin{cor}
\label{cor:21}
    Let $\mathcal T\in \Omega$ and $u \in \overline{\mathcal T}$.
    Then $\Gamma[N(u) \cap \mathcal T] = 3K_2 + 4K_1$.
\end{cor}
\begin{proof}
Let $\mathbf u = A[\mathcal T,\{u\}]$ and set $\alpha = \mathbf u^\transpose A[\mathcal T]^2 \mathbf u$ and $\beta = \mathbf u^\transpose  A[\mathcal T] \mathbf u$.
By Lemma~\ref{lem:21}, we have $\alpha = 28$, $\beta = 6$, and $\mathbf 1^\transpose \mathbf u = 10$.
By Lemma~\ref{lem:5squares}, we can write $\mathcal T = \bigcup_{i=1}^5 \mathcal C_i$ where $\Gamma[\mathcal C_i] = C_4$ for each $i \in \{ 1,\dots,5\}$.
Observe that
\begin{align}
    \sum_{i=1}^5 (\mathbf 1^\transpose \mathbf u[\mathcal C_i])^2 = \sum_{i=1}^5 \mathbf u[\mathcal C_i]^\transpose J \mathbf u[\mathcal C_i] &= \sum_{i=1}^5 \mathbf u[\mathcal C_i]^\transpose \left(\frac{A[\mathcal C_i]^2}{2}+A[\mathcal C_i]\right) \mathbf u[\mathcal C_i] \nonumber \\
    &= \sum_{i=1}^5 \mathbf u[\mathcal C_i]^\transpose \frac{A[\mathcal C_i]^2}{2} \mathbf u[\mathcal C_i] + \sum_{i=1}^5 \mathbf u[\mathcal C_i]^\transpose A[\mathcal C_i] \mathbf u[\mathcal C_i] \nonumber \\
    &= \frac{\alpha}{2} + \beta = 20. \nonumber
\end{align}
Combining with $\mathbf 1^\transpose \mathbf u = 10$, we must have $\mathbf 1^\transpose \mathbf u[\mathcal C_i] = 2$ for each $i \in \{1,\dots,5\}$.
Next, observe that, for each $i \in \{1,\dots 5\}$, we have $\mathbf u[\mathcal C_i]^\transpose A[\mathcal C_i] \mathbf u[\mathcal C_i] \in \{0,2\}$.
Therefore,  $\mathbf u^\transpose A[\mathcal T] \mathbf u = \sum_{i=1}^5 \mathbf u[\mathcal C_i]^\transpose A[\mathcal C_i] \mathbf u[\mathcal C_i] = 6$ implies that $|\{ i \in \{1,\dots,5\} \; : \; \mathbf u[\mathcal C_i]^\transpose A[\mathcal C_i] \mathbf u[\mathcal C_i] = 2 \}| = 3$ and  $|\{ i \in \{1,\dots,5\} \; : \; \mathbf u[\mathcal C_i]^\transpose A[\mathcal C_i] \mathbf u[\mathcal C_i] = 0 \}| = 2$.
\end{proof}

By Corollary~\ref{cor:21}, for $\mathcal T \in \Omega$ and $u \in \overline{\mathcal T}$, the induced subgraph $\Gamma[\mathcal T \cup \{u\}]$ is isomorphic to $\mathfrak{G}$ in Figure~\ref{fig:graphG}.
Partition $\mathcal T$ as $\mathcal T = \bigcup_{i=1}^2 \mathcal C_i \cup \bigcup_{i=1}^3 \mathcal C_i^\star$ such that $\Gamma[\mathcal C_i] = C_4$, $\Gamma[\mathcal C_i^\star] = C_4$, $\Gamma[N(u) \cap \mathcal C_i] = 2K_1$ for $i \in \{1,2\}$, and $\Gamma[N(u) \cap \mathcal C_i^\star] = K_2$ for $i \in \{1,2,3\}$.
This way, we see that there are $2560 = \binom{5}{2}\cdot 2^2 \cdot 4^3$ subsets in $\binom{\mathcal T}{10}$ that can be used as the neighbourhood of $u$ to obtain $\mathfrak G = \Gamma[\mathcal T \cup \{u\}]$.

Subsequently, we consider graphs resulting from joining an independent set of order two to $\mathcal{T}$.
The next lemma provides conditions on the adjacency of the independent set with the vertices in $\mathcal{T}$.
\begin{lem}
\label{lem:22}
    Let $\mathcal T,\mathcal U \in \Omega$ and $u,v \in \mathcal U$ where $\mathcal T \ne \mathcal U$, $u \ne v$, and $u \not \sim v$.
    Suppose that $A$ is the adjacency matrix of $\Gamma$ and $A[\mathcal T,\{u,v\}] = [\mathbf u, \mathbf v]$.
    Then
    \[
    \left ( \mathbf u^\transpose \mathbf v, \mathbf u^\transpose A[\mathcal T]\mathbf v  \right ) \in \{ (4,14), (5,10), (6,6) \}.
    \]
\end{lem}
\begin{proof}
Let $B = A[\mathcal T,\{u,v\}] = [\mathbf u, \mathbf v]$.
Define $\alpha = \mathbf u^\transpose \mathbf v$ and $\beta = \mathbf u^\transpose A[\mathcal T] \mathbf v$.
Then, by Lemma~\ref{lem:21},
\[
B^\transpose B = \begin{bmatrix}10 & \alpha \\ \alpha & 10 \end{bmatrix} \quad \text{ and } \quad B^\transpose A[\mathcal T] B = \begin{bmatrix}6 & \beta \\ \beta & 6 \end{bmatrix}.
\]
Thus, the determinant of
\begin{align*}
       & (x+4)I+\frac{10(x^2 + 4x-30)J}{(x-12)(x+6)(x+8)}-\frac{B^\transpose B}{x+4}-\frac{40J}{(x+2)(x+4)(x+6)}+\frac{B^\transpose A[\mathcal T]B}{(x+2)(x+4)} 
       \end{align*}
is equal to
\[
\frac{f(x)g(x)}{(x-12)(x+2)^2(x+4)^2(x+8)},
\]
where 
\begin{align*}
    f(x) &= x^5+  6 x^4-(\alpha+94) x^3 +(2\alpha+\beta-950) x^2 +(104\alpha-4\beta-2704) x +192\alpha-96\beta-1248, \\
    g(x) &= x^3+10x^2+(\alpha+22)x+2\alpha-\beta+18.
\end{align*}
By Lemma~\ref{lem:decaenBlock}, we must have $f(x)g(x) \in (x+6)^2 \mathbb Z[x]$.
Reducing modulo $(x+6)^2$ yields the simultaneous equations
\begin{align*}
    384\alpha^2+132\beta^2+624\alpha\beta-12000 \alpha-6240\beta+68400 &= 0, \\
    32\alpha^2-16\beta^2-56\alpha\beta+560\alpha+680\beta-6000 &= 0.
\end{align*}
It follows that $\beta+4\alpha = 30$.

Let $\mathbf w = A[\mathcal T]\mathbf u$.
By Lemma~\ref{lem:5squares}, we can partition the set $\mathcal T$ as $\mathcal T = \bigcup_{i=1}^5 \mathcal C_i$ such that $\Gamma[\mathcal C_i] = C_4$ for each $i \in \{1,\dots,5\}$.
By Corollary~\ref{cor:21}, there are two indices, $a$ and $b$, such that $\Gamma[N(u) \cap \mathcal C_i] = 2K_1$ for each $i \in \{a,b\}$.
It follows that
\[
 \mathbf w(w) = \begin{cases}
 0 & \text{ if $w \in N(u) \cap (\mathcal C_a \cup \mathcal C_b)$}; \\
 1 & \text{ if $w \in \mathcal T \backslash (\mathcal C_a \cup \mathcal C_b)$}; \\
 2 & \text{ otherwise}. \\
 \end{cases}
\]
Since $|N(v) \cap \mathcal C_i| = 2$ for each $i \in \{1,\dots,5\}$, it follows that $\mathbf w^\transpose \mathbf v \in \{6,10,14\}$.
Furthermore, since $\beta = \mathbf w^\transpose \mathbf v$ and $\beta+4\alpha = 30$, we must have $(\alpha,\beta) \in \{ (4,14), (5,10), (6,6) \}$.
\end{proof}

Fix $\mathcal T, \mathcal U \in \Omega$ with $\mathcal T \ne \mathcal U$ and $u \in \mathcal U$.
By Corollary~\ref{cor:21}, the induced subgraph $\Gamma[\mathcal T \cup \{u\}]$ is isomorphic to $\mathfrak{G}$ in Figure~\ref{fig:graphG}.
For each subset $\mathcal K \in \binom{\mathcal T}{10}$, define the corresponding (characteristic) vector $\mathbf v_{\mathcal K} \in \{0,1\}^{\mathcal T}$ such that $\mathbf v_{\mathcal K}(w) = 1$ if $w \in \mathcal K$ and $\mathbf v_{\mathcal K}(w) = 0$ otherwise.
Let $\mathbf u$ denote $\mathbf v_{N(u) \cap \mathcal T}$.
Define the set $\mathfrak B_u \subset \{0,1\}^{\mathcal T}$ to consist of vectors $\mathbf v$ such that 
\begin{itemize}
    \item $\mathbf v^\transpose \mathbf v = 10$.
    \item $\mathbf v^\transpose A[\mathcal T] \mathbf v = 6$.
    \item $\mathbf v^\transpose A[\mathcal T]^2 \mathbf v = 28$.
    \item $(\mathbf u^\transpose \mathbf v,\mathbf u^\transpose A[\mathcal T] \mathbf v) \in \{(4,14),(5,10),(6,6)\}$.
\end{itemize}

By Lemma~\ref{lem:21} and Lemma~\ref{lem:22}, for any $v \in \mathcal U$ with $v \not \sim u$, we must have $\mathbf v_{N(v) \cap \mathcal T} \in \mathfrak B_u$.
Form a graph $G_u$ with vertex set $\mathfrak B_u$ where $\mathbf v \in \mathfrak B_u$ and $\mathbf w \in \mathfrak B_u$ are adjacent if and only if
\[
(\mathbf v^\transpose \mathbf w,\mathbf v^\transpose A[\mathcal T] \mathbf w) \in \{(4,14),(5,10),(6,6)\}.
\]

Note that the graph $G_u$ does not depend on the choice of $u \in \mathcal U$.
For each $u^\star \in \mathcal U$, since $\Gamma[\mathcal T \cup \{u^\star\}] = \Gamma[\mathcal T \cup \{u\}] = \mathfrak G$, there exists a permutation matrix $P$ such that $\mathbf v_{N(u^\star) \cap \mathcal T} = P\mathbf u$ and $P^\transpose A[\mathcal T]P = A[\mathcal T]$.
Thus, $\mathbf v \in \mathfrak B_u$ if and only if $P\mathbf v \in \mathfrak B_{u^\star}$ and furthermore, $G_u$ is isomorphic to $G_{u^\star}$.

We remark that the graph $G_u$ has 454 vertices.
This can be proved without the use of a computer using Lemma~\ref{lem:22}, but we omit it since our proof is a bit tedious and the graph $G_u$ can be readily generated using a computer.

By Lemma~\ref{lem:5squares}, the induced subgraph $\Gamma[\mathcal U]$ is the disjoint union of five $4$-cycles, $5C_4$.
This implies that there exists a subset of $10$ pairwise nonadjacent vertices in $\mathcal U$ including the vertex $u$.
This subset corresponds to a clique of order $9$ in the graph $G_u$.
However, using Magma (and independently, Mathematica), we find that the largest clique in $G_u$ has $7$ vertices.
This contradicts our assumption that the graph $\Gamma$ exists and completes the proof of Theorem~\ref{thm:last5_3ev}.

\section{Acknowledgements}

The first author has benefited from insightful conversations with Jack Koolen and Akihiro Munemasa.
The second author has benefited from advice from Chris Godsil.
The first author was supported by the Singapore
Ministry of Education Academic Research Fund (Tier 1);  grant numbers:  RG21/20 and RG23/20.

\appendix 

\section{Tables of certificates of infeasibility}

In Table~\ref{tab:39polydim18}, we list all but five of the candidate characteristic polynomials from Theorem~\ref{thm:candpols} together with their certificates of infeasibility. 

\begin{longtable}{l}

\endfirsthead

\multicolumn{1}{c}
{{\bfseries \tablename\ \thetable{} -- continued from previous page}} \\ \hline 
\endhead

\hline \multicolumn{1}{r}{{Continued on next page}} \\ \hline
\endfoot

\endlastfoot

          1.\;\;\: $(x+5)^{42} (x-11)^{12} (x-13)^3 (x^3-39x^2+495x-2049)$   \\
          \qquad $(0, 0, 0, 34294, 10143, 1812)$ \\
          \hline
          
          2.\;\;\: $(x+5)^{42} (x-11)^{13} (x-13)^2 (x^3-41x^2+551x-2423)$   \\
          \qquad $(-305901333, 0, 0, -29631, -1854, 35)$ \\
          \hline
          
          3.\;\;\: $(x+5)^{42} (x-11)^{14} (x-13) (x-15) (x^2-28x+191)$   \\
          \qquad $(-1458935482, 0, 0, -81026, 9, 0)$ \\
          \hline
          
          4.\;\;\: $(x+5)^{42} (x-11)^{12} (x-13)^4 (x^2-26x+157)$   \\
          \qquad $(0, 0, 0, -843, -338)$ \\
          \hline
          
          5.\;\;\: $(x+5)^{42} (x-11)^{12} (x-13)^2 (x^2-24x+139) (x^2-28x+191)$   \\
          \qquad $(-7692466694472, 0, 0, -458256882, -36586502, -2928267, -235607)$ \\
          \hline
          
          6.\;\;\: $(x+5)^{42} (x-11)^{13} (x-13)^2 (x-15) (x^2-26x+161)$   \\
          \qquad $(0, 0, 0, 0, 370, 171)$ \\
          \hline
          
          7.\;\;\: $(x+5)^{42} (x-9) (x-11)^{11} (x-13)^4 (x^2-28x+191)$   \\
          \qquad $(19949599256, 0, 0, 2370351, 198001, 16971)$ \\
          \hline
          
          8.\;\;\: $(x+5)^{42} (x-11)^{12} (x-13)^3 (x^3-39x^2+495x-2033)$   \\
          \qquad $(-311392576602, 0, 0, -31763185, -2280433, -162889)$ \\
          \hline
          
          9.\;\;\: $(x+5)^{42} (x-11)^{11} (x-13)^3 (x^2-24x+139) (x^2-26x+161)$   \\
          \qquad $(0, 0, 0, 0, 0, -1389, -790)$ \\
          \hline
          
          10.\; $(x+5)^{42} (x-11)^{12} (x-13)^2 (x-15) (x^3-37x^2+447x-1763)$   \\
          \qquad $(0, 0, 0, 0, -122674, -43601, -8823)$ \\
          \hline
          
          11.\; $(x+5)^{42} (x-11)^{13} (x-13) (x-15)^2 (x^2-24x+139)$   \\
          \qquad $(0, 0, 0, 0, 420, 199)$ \\
          \hline
          
          12.\; $(x+5)^{42} (x-11)^{11} (x-13)^4 (x^3-37x^2+443x-1711)$   \\
          \qquad $(0, 0, 0, 234708, 62958, 8869)$ \\
          \hline
          
          13.\; $(x+5)^{42} (x-9) (x-11)^{10} (x-13)^5 (x^2-26x+161)$   \\
          \qquad $(0, 0, 0, 125208, 41635, 8550)$ \\
          \hline
          
          14.\; $(x+5)^{42} (x-11)^{11} (x-13)^3 (x^4-50x^3+924x^2-7470x+22259)$   \\
          \qquad $(0, 0, 0, -24742927, -5579245, -747551, -85438)$ \\
          \hline
          
          15.\; $(x+5)^{42} (x-11)^{10} (x-13)^3 (x^2-24x+139) (x^3-37x^2+447x-1763)$   \\
          \qquad $(0, 0, 0, 0, 0, 0, 2249, 1485)$ \\
          \hline
          
          16.\; $(x+5)^{42} (x-9) (x-11)^{12} (x-13)^3 (x-15)^2$   \\
          \qquad $(-196837694, 0, 0, -55815, -5074)$ \\
          \hline
          
          17.\; $(x+5)^{42} (x-11)^{11} (x-13)^2 (x-15) (x^2-24x+139)^2$   \\
          \qquad $(0, 0, 0, -172391, -37294, -2664)$ \\
          \hline
          
          18.\; $(x+5)^{42} (x-11)^{11} (x-13)^5 (x^2-24x+131)$   \\
          \qquad $(0, 0, 0, -2265, -946)$ \\
          \hline
          
          19.\; $(x+5)^{42} (x-9) (x-11)^9 (x-13)^5 (x^3-37x^2+447x-1763)$   \\
          \qquad $(0, 0, 0, -34871227, -8914066, -1403713, -190406)$ \\
          \hline
          
          20.\; $(x+5)^{42} (x-11)^{10} (x-13)^4 (x^4-48x^3+850x^2-6576x+18749)$  \\
          \qquad $(0, 0, 0, 0, 0, -176, -103)$ \\
          \hline
          
          21.\; $(x+5)^{42} (x-11)^{11} (x-13)^3 (x-15) (x^3-35x^2+399x-1477)$   \\
          \qquad $(0, 0, 0, 0, -94177, -32024, -5780)$ \\
          \hline
          
          22.\; $(x+5)^{42} (x-9) (x-11)^{10} (x-13)^4 (x-15) (x^2-24x+139)$   \\
          \qquad $(0, 0, 0, 0, -557458, -207778, -43657)$ \\
          \hline
          
          23.\; $(x+5)^{42} (x-11)^9 (x-13)^3 (x^2-24x+139)^3$   \\
          \qquad $(0, 0, 0, -415, -169)$ \\
          \hline
          
          24.\; $(x+5)^{42} (x-11)^{10} (x-13)^5 (x^3-35x^2+395x-1433)$   \\
          \qquad $(27648369302, 0, 0, 3643351, 310262, 26593)$ \\
          \hline
         
          25.\; $(x+5)^{42} (x-11)^9 (x-13)^5 (x^4-46x^3+780x^2-5778x+15779)$   \\
          \qquad $(0, 0, 0, 68304472, 16503986, 2380728, 292487)$ \\
          \hline
         
          26.\; $(x+5)^{42} (x-11)^{11} (x-13)^4 (x-15) (x^2-22x+113)$   \\
          \qquad $(0, 0, 0, 108251, 29388, 4254)$ \\
          \hline
         
          27.\; $(x+5)^{42} (x-11)^9 (x-13)^4 (x^2-24x+139) (x^3-35x^2+399x-1477)$   \\
          \qquad $(0, 0, 0, -1845487303, -379343200, -48329261, -5293439, -539442)$ \\
          \hline
         
          28.\; $(x+5)^{42} (x-9) (x-11)^8 (x-13)^5 (x^2-24x+139)^2$   \\
          \qquad $(34043665264, 0, 0, 5251934, 519994, 51999)$ \\
          \hline
         
          29.\; $(x+5)^{42} (x-11)^{10} (x-13)^5 (x-15) (x^2-20x+95)$   \\
          \qquad $(64088824162, 0, 0, 8048870, 611903, 47069)$ \\
          \hline
         
          30.\; $(x+5)^{42} (x-11)^9 (x-13)^6 (x^3-33x^2+351x-1207)$   \\
          \qquad $(0, 0, 0, 0, 1139, 599)$ \\
          \hline
         
          31.\; $(x+5)^{42} (x-9) (x-11)^8 (x-13)^6 (x^3-35x^2+399x-1477)$   \\
          \qquad $(0, 0, 0, 0, -814967, -316527, -68325)$ \\
          \hline
         
          32.\; $(x+5)^{42} (x-11)^9 (x-13)^5 (x^2-22x+113) (x^2-24x+139)$   \\
          \qquad $(0, 0, 0, 0, 0, -1163, -700)$ \\
          \hline
         
          33.\; $(x+5)^{42} (x-9)^2 (x-11)^7 (x-13)^7 (x^2-24x+139)$   \\
          \qquad $(0, 0, 0, 0, 248, 129)$ \\
          \hline
         
          34.\; $(x+5)^{42} (x-11)^8 (x-13)^6 (x^2-20x+95) (x^2-24x+139)$   \\
          \qquad $(531003714336, 0, 0, 41982177, 3283852, 255629, 19664)$ \\
          \hline
         
          35.\; $(x+5)^{42} (x-9) (x-11)^8 (x-13)^7 (x^2-22x+113)$   \\
          \qquad $(-24681212876, 0, 0, -4267959, -431106, -43111)$ \\
          \hline
         
          36.\; $(x+5)^{42} (x-11)^9 (x-13)^6 (x^3-33x^2+351x-1191)$   \\
          \qquad $(-6730490844, 0, 0, -798929, -33284, 0)$ \\
          \hline
         
          37.\; $(x+5)^{42} (x-11)^8 (x-13)^7 (x^3-31x^2+311x-1009)$   \\
          \qquad $(0, 0, 0, 0, 1329, 722)$ \\
          \hline
         
          38.\; $(x+5)^{42} (x-9) (x-11)^7 (x-13)^8 (x^2-20x+95)$   \\
          \qquad $(-816690681, 0, 0, -119057, -4880, 14)$ \\
          \hline
         
          39.\; $(x+5)^{42} (x-7) (x-11)^9 (x-13)^8$   \\
          \qquad $(143620, 0, 0, 253)$ \\
          \hline
         
          \caption{Candidate characteristic polynomials of Theorem~\ref{thm:candpols} and their certificates of infeasibility.}
    \label{tab:39polydim18}
\end{longtable}

\end{document}